\documentclass[11pt]{amsart}
\usepackage{graphicx}
\usepackage{amssymb}
\usepackage{epstopdf}

\usepackage{amsmath,amsfonts,amsthm,amsopn,cite,mathrsfs}

 \theoremstyle{theorem}
\newtheorem{lemma}{Lemma}
\newtheorem{theorem}{Theorem}[section]
\newtheorem{corollary}[theorem]{Corollary}

\newtheorem{proposition}[theorem]{Proposition}

\theoremstyle{remark}

\newcommand{\C}{\mathbb{C}}

\newcommand{\T}{\mathbb T}
\newcommand{\R}{\mathbb R}
\newcommand{\Z}{\mathbb Z}

\def\S{{\mathcal S}}
\def\A{{\mathcal A}}

\def\T{{\mathbb T}}

\newcommand{\veps}{\varepsilon}

\newcommand{\norm}[1]{{\left\|{#1}\right\|}} 
\newcommand{\beq}{\begin{equation}}
\newcommand{\eeq}{\end{equation}}

\newcommand{\scal}[1]{{\left\langle{#1}\right\rangle}}

\begin{document}

\title[Discrete  Schr\"odinger Evolutions]{ Uniqueness for Discrete Schr\"{o}dinger Evolutions}

\author[Ph. Jaming]{Philippe Jaming}
\address{Institut de Math\'ematiques de Bordeaux UMR 5251,
Universit\'e de Bordeaux, cours de la Lib\'eration, F 33405 Talence cedex, France}
\email{philippe.jaming@math.u-bordeaux1.fr}

\author[Yu. Lyubarskii]{Yurii Lyubarskii}
\address{Department of Mathematics,
Norwegian University of Science and Technology
7491, Trondheim, Norway}
\email{yura@math.ntnu.no}

\author[E. Malinnikova]{Eugenia Malinnikova}
\email{eugenia@math.ntnu.no}

\author[K.-M. Perfekt]{Karl-Mikael Perfekt}
\email{karl-mikael.perfekt@math.ntnu.no}

\thanks{
Ph.J. kindly acknowledge financial support from the French ANR programs ANR
2011 BS01 007 01 (GeMeCod), ANR-12-BS01-0001 (Aventures).
This study has been carried out with financial support from the French State, managed
by the French National Research Agency (ANR) in the frame of the ”Investments for
the future” Programme IdEx Bordeaux - CPU (ANR-10-IDEX-03-02).\\
E.M.  and Yu.L. were supported by Project 213638 of the Research Council of Norway.\\
This research was sponsored by the French-Norwegian PHC AURORA 2014 PROJECT N° 31887TC, N~233838, {\it CHARGE}} 
\begin{abstract} We prove that if a solution of the discrete time-dependent Schr\"odinger equation with bounded real potential decays fast at two distinct times then the solution is trivial. For the free Shr\"odinger operator and for operators with compactly supported time-independent potentials a sharp analog of the Hardy uncertainty principle is obtained, using an argument based on the theory of entire functions. Logarithmic convexity of weighted norms is employed in the case of general real-valued time-dependent bounded potentials. In the latter case the result is not optimal.    
\end{abstract}

\maketitle

\section{Introduction}

The aim of this paper is to show that a non-trivial solution of a semi-discrete Shr\"odinger equation with bounded real potential cannot have arbitrarily fast decay at two different times. For the free evolution (with no potential) the result we obtain is precise and it can be interpreted as a discrete version of the dynamical Hardy Uncertainty Principle.

The usual formulation of the Uncertainty Principle is that a function and its Fourier transform 
can not both be arbitrarily well localized. In Hardy's 
Uncertainty Principle, the localization is measured in terms of speed of decay at infinity:
{\sl if $f \in L^2(\mathbb{R})$ is such that $f$ and its Fourier transform $\hat{f}$ satisfy
\[
|f(x)|\le C\exp(-\pi |x|^2), \quad |\hat {f}(\xi)|\le C\exp(-\pi |\xi|^2), \quad x, \xi \in \mathbb{R},
\] 
for some constant $C > 0$, then there is a constant $A$ such that $f(x)=A\exp(-\pi |x|^2)$.}

It is known 
that Uncertainty Principles may also be given dynamical interpretations in terms of solutions of the free Schr\"odinger Equation \cite{Ja,EKPV,F}. Hardy's Uncertainty Principle is equivalent to the following statement:\\ 
($\ast$) {\sl if $u(t,x)$ is a solution of the free Schr\"odinger equation $\partial_t u=i\Delta u$ and $|u(0,x)|+|u(1,x)|\le C\exp(-x^2/4)$, then $u(0,x)=A\exp(-(1+i)x^2/4)$.}\\
The point is that the free Shr\"odinger Equation can be explicitly solved via the Fourier transform, from which the two formulations of the Hardy Uncertainty Principle are seen to be equivalent. 

In a remarkable series of papers, 
L. Escauriaza, C. E. Kenig, G. Ponce and L.~Vega \cite{EKPV0,EKPV,EKPV1,EKPV2}, and also in collaboration with Cowling \cite{CEKPV}, have extended the uniqueness statement ($\ast$) to solutions of Schr\"odinger equations with potentials, as well as to solutions of a wider class of partial differential equations that even includes some non-linear equations. Further results for covariant Schr\"odinger evolutions were obtained in \cite{BFGRV,CF}. Concerning the discrete setting, we note that a discrete dynamical interpretation of the Heisenberg uncertainty principle was given in \cite{F}.
 
In the present work we obtain uniqueness results for solutions of the discrete time-dependent Schr\"{o}dinger equation
\begin{equation}
 \label{eq:schrod}
\partial_t u=i(\Delta_d u+Vu),
\end{equation}
where $u:\R_+\times\Z\to\C$ and the potential $V=V(t,n)$ is a real-valued bounded function. $\Delta_d $  is the discrete Laplacian; for a complex valued function $f:\Z\to\C$, 
\[
\Delta_d f(n):=f(n+1)+f(n-1)-2f(n).
\]
We say that $u = u(t,n)$ is a strong solution of \eqref{eq:schrod} if $u \in C^1([0,\infty),\ell^2)$. In the case that $V = V(n)$ is time-independent, $H=\Delta_d+V$ is a bounded self-adjoint operator in $\ell^2$, and for any $u(0,\cdot)\in \ell^2$ there exists a unique strong solution of  \eqref{eq:schrod} with $u(0,\cdot)$ as the initial   value:
$u(t,\cdot)=e^{itH}u(0,\cdot)$. If the potential $V$ is time-dependent but real, the $\ell^2$-norm of a solution is preserved.

For the free evolution we obtain what can be considered the discrete analog of the dynamical interpretation of Hardy's Uncertainty Principle, which we then extend to equations with bounded real potentials. The results bear similarities to the continuous case, but there are at the same time fundamental differences. For instance, in the setting of free evolution, in both the continuous and the discrete case the optimal decay is given by the heat kernel at time $1$.  However, this means that the critical decay is different for the two situations.  For the continuous case the standard heat kernel is $k(1,0,x)=(4\pi)^{-1/2}\exp(-x^2/4)$, while for the discrete case the heat kernel is $K(1,0,n)=e^{-1}|I_n(1)|\asymp e^{-1}(n!2^n)^{-1}$, where $I_n$ are the modified Bessel functions, $I_n(z)=(-i)^nJ_n(iz)$. Computations of the discrete heat kernel for the lattice and asymptotic connecting the two cases can be found in\cite{CY1,CY2}.  Discrete heat kernels also appeared as weights for convexity results for discrete harmonic functions in recent work by G. Lippner and D. Mangoubi \cite{LM}.

To finish the introduction we now describe our main results in greater detail. 
First, we prove that
if $u(t,n)$ solves the free equation 
$$
\partial_t u = i \Delta_d u
$$ 
and satisfies 
\begin{equation} \label{eq:schroddec}
|u(0,n)|+|u(1,n)|\le C \frac{1}{\sqrt{|n|}}\left(\frac{e}{2|n|}\right)^{|n|},\quad n\in\Z\setminus\{0\},
\end{equation}
then $u(t,n)=Ai^{-n} e^{-2it}J_n(1 - 2t)$, where $J_n$ is the Bessel function. This result is sharp:   $|J_n(-1)|$ and $|J_n(1)|$ have precisely the growth of the right hand side in  \eqref{eq:schroddec} as $|n| \to \infty$, see Proposition \ref{prop:free}.

Then we investigate the discrete equation \eqref{eq:schrod} with bounded potential in essentially two different ways using techniques of complex and real analysis.  First, applying the theory of entire functions, we establish that if $V(n)$ is a compactly supported potential and $u$ is a strong solution of \eqref{eq:schrod} satisfying the one-sided estimates
\begin{equation*}
|u(t,n)|\le C\left(\frac{e}{(2+\epsilon)n}\right)^n,\quad  n>0,\ t\in\{0,1\},
\end{equation*} 
for some $\epsilon > 0$, then $u \equiv 0$. In the continuous setting, one-sided Hardy uncertainty principles have previously appeared in works of Nazarov \cite{N} and Demange \cite{D}. The corresponding results for continuous Schr\"odinger evolutions can be found in the recent survey \cite{EKPV2}.

In the second part of the paper, we  use the real-variable approach following  \cite{EKPV}. The main idea is to construct a weight function $\psi(t,n)$ which provides the  logarithmic convexity  of the weighted $\ell^2$ norms $\|\psi(t,\cdot) u(t,\cdot)\|_{\ell^2(\mathbb{Z})}$, where $u(t,n)$ is a strong solution of \eqref{eq:schrod}.  This line of reasoning has its roots in celebrated results of T. Carleman and S. Agmon; the technique of Carleman estimates goes back to \cite{C} and convexity principles for elliptic operators were described in \cite{A}.
The method allows us to consider  general bounded potentials $V$, at the cost of having to assume stronger decay of $u(0,n)$  and $u(1,n)$  in both directions $n\to \pm \infty$.  The main result, Theorem \ref{thm:main}, says that if 
\[
\norm{(1+|n|)^{\gamma(1+|n|)}u(0,n)}_{\ell^2(\mathbb{Z})}, \,\norm{(1+|n|)^{\gamma(1+|n|)}u(1,n)}_{\ell^2(\mathbb{Z})}< \infty
\]
for some $\gamma > (3+\sqrt{3})/2$, then $u \equiv 0$.
We don't expect this result to be sharp, but it does provide a universal decay condition which implies uniqueness of solutions of Schr\"odinger equations with general bounded potentials.

The paper is organized as follows: in Section 2 we discuss preliminaries of entire functions and use them to obtain our first results. Section 3 contains a precursory energy estimate for solutions of \eqref{eq:schrod}, which we need in order to justify the validity of many of our computations. Section 4  splits into several subsections discussing and proving the logarithmic convexity results we require, and in the final subsection the main result is proven.

Note that we will use the symbol $C$ to denote various constants in what follows. 
Unless otherwise indicated, its value might change from line to line. Also all $\norm{\, \cdot \,}_2$-norms are to be understood as the $\ell^2$-norm in the variable $n$.

\section{A uniqueness result for Schr\"{o}dinger operators with compactly supported potentials}

In this section, we use methods from complex analysis. For the reader's convenience,
we begin by briefly outlining some definitions and facts on entire functions of exponential
type that we need. Details can be found in \cite{Lbook} (see in particular Lectures 8 and 9).
Recall that an entire function $f$ is said to be of exponential type if for some $k>0$
\begin{equation}\label{eq:exptype}
|f(z)|\le C\exp(k|z|).
\end{equation}  
In this case the type of an entire function $f$ is defined by 
\begin{equation}
\sigma =\limsup_{r\to\infty}\frac{\log \max\{|f(re^{i\phi})|;\phi\in[0,2\pi]\}}{r} <\infty.
\end{equation}
In particular, an entire function $f$ is of zero exponential type if for any $k>0$ there exists $C=C(k)$ such that \eqref{eq:exptype} holds.

Let $f(z)$ be an entire function of exponential type, $f(z)=\sum_{n=0}^\infty c_nz^n$. Then the type of $f$ can be  expressed in terms of the Taylor coefficients in the following way
\begin{equation}\label{eq:coeftype}
 \limsup_{n\to\infty} n|c_n|^{1/n} = e\sigma. 
 \end{equation}

The growth of a function $f$ of exponential type along different directions is described  by the indicator function
\[
h_f(\varphi)=\limsup_{r\to\infty}\frac{\log|f(re^{i\varphi})|}{r}.
\]
This function is the support function of some convex compact set $I_f\subset \C$ which is called the indicator diagram of $f$. In particular
\beq\label{eq:if}
h_f(\varphi)+h_f(\pi-\varphi)\ge 0.
\eeq 
For example the indicator function of $e^{az}$ for $a\in \C$ is $h(\varphi)=\Re(ae^{i\varphi})$ and its indicator diagram consists of 
a single  point, $\bar{a}$. 

Clearly, $h_{fg}(\varphi)\le h_f(\varphi)+h_g(\varphi)$, implying that
\[
I_{fg}\subset I_f+I_g:=\{z=z_1+z_2: z_1\in I_f, z_2\in I_g\}.
\] 

Recall that   the Bessel functions $J_n$ satisfy
\[
\exp(x(z-z^{-1})/2)=\sum_{n=-\infty}^\infty J_n(x)z^n, \quad z \neq 0.
\]
Moreover, for fixed $x$, 
$$
|J_n(x)|\sim \frac{1}{\sqrt{|n|}} \left( \frac{ex}{2|n|}\right)^{|n|} \ \mbox{as} \   |n| \to \infty.
$$

Our first observation is the following discrete analog of the classical Hardy uncertainty principle.
\begin{proposition} \label{prop:free}
Let $u\in C^1([0,1],\ell^2)$ satisfy the discrete free Schr\"{o}dinger equation $\partial_t u=i\Delta_d u$,   and suppose that 
\begin{equation}
\label{eq:2f}
|u(0,n)|, \ |u(1,n)|\le C \frac{1}{\sqrt{|n|}}\left(\frac{e}{2|n|}\right)^{|n|},\quad n\in\Z\setminus\{0\}
\end{equation}
for some $C > 0$.  
Then 
$u(t,n)=Ai^{-n} e^{-2it}J_n(1 - 2t)$ for all $n\in\Z$ and $0\le t\le 1$, for some constant $A$.

\end{proposition}

\begin{proof}
Consider the discrete Fourier transforms of $u(t,\cdot)$,
$$
\Phi(t,\theta)=\sum_{k=-\infty}^\infty u(t,k)\theta^{k} \ \in L^2(\T).
$$
We have  $\partial_t \Phi(t,\theta)=i(\theta+\theta^{-1}-2)\Phi(t,\theta)$.
Thus 
$$
\Phi(t,\theta)=e^{i(\theta+\theta^{-1}-2)t}\Phi(0,\theta),
$$ 
 and in particular 
\begin{equation}
\label{eq:01}
\Phi(1,\theta)=e^{i(\theta+\theta^{-1}-2)}\Phi(0,\theta).
\end{equation}
It follows from \eqref{eq:2f} that the functions $\theta\mapsto\Phi(s,\theta)$, for $s=0$ and $s=1$, admit holomorphic extensions to $\C\setminus\{0\}$:
\begin{equation}
\label{eq:extension}
\Phi(s,\theta)=\sum_{k<0}u(k,s)\theta^k + \sum_{k\geq 0}u(k,s)\theta^k = \Phi^-(s,\theta) +  \Phi^+(s,\theta), \ s\in\{0,1\}.
\end{equation}
Furthermore, \eqref{eq:2f} implies that  $ \Phi^+(s,\theta)$ and $ \Phi^-(s,1/\theta)$, $s=0$ and $s=1$,  are entire functions of exponential type 
whose indicator diagrams $I_s^+$ and $I_s^-$, respectively,  are contained in the disk of radius $1/2$ centered at zero. Actually one can say more: 
\begin{equation}
\label{eq:01a}
| \Phi^+(s,\theta)|, \  | \Phi^-(s,1/\theta) |  \leq C e^{|\theta|/2},\quad s\in\{0,1\}.
\end{equation}
This follows from the fact that the right-hand side of \eqref{eq:2f} is asymptotically equivalent to the coefficients in the Taylor expansion of
$\exp(z/2)$. On the other hand it follows from \eqref{eq:01} that  $I_1^\pm \subset I_0^\pm+i$. Thus $I_0^\pm=\{-i/2\}$ and $I_1^\pm=\{i/2\}$.

Now let
\begin{equation}
\label{eq:01b}
g(z)=g^+(z)+g^-(z)= e^{i(z+z^{-1})/2}\Phi(0,z)=e^{-i(z+z^{-1})/2}\Phi(1,z),
\end{equation}
where, as before, $g^\pm$ are the parts  of the  Laurent  series of $g$ with respectively non-negative and negative powers.    
It follows that the indicator diagrams $I^\pm$ of the entire functions $g^+(z)$ and $g^-(1/z)$ coincide with $\{0\}$, so 
$g^+(z)$ and $g^-(1/z)$ are entire functions of type zero.  

The relations \eqref{eq:01a} and \eqref{eq:01b} now yield that $g^+(iy)$ and $g^-(1/iy)$ are bounded for $y\in \R\setminus \{0\}$ and by the  Phragmen-Lindel\"{o}f  principle (see \cite{Lbook}, Lecture 6)  $g^+$, $g^-$, and hence   $g$,  are  constants.
Finally  $\Phi_0(z)=A\exp(-i(z+z^{-1})/2)$, yielding the required expression  for $u(t,n)$.
\end{proof}

\begin{corollary}
Let  $u$ be as in Proposition \ref{prop:free} if in addition  
\[|u(0,n)|\left(\frac{2|n|}{e} \right)^{|n|}\sqrt{|n|}=o(1)\] as $n\to+\infty$ or $n\to -\infty$ then $u\equiv 0$.
\end{corollary}

Assuming only slightly stronger decay,   one can apply  similar techniques  in order to obtain a uniqueness result 
for  solutions  of discrete Schr\"{o}dinger equations with compactly supported time-independent potentials. In this case 
it suffices to demand that the solution decays just in one direction.
\begin{theorem} \label{th:1} Let $u(t,n)$, $t>0$, $n\in \Z$ be a solution of \eqref{eq:schrod},   where the potential $V$   does not depend on time and also $V(n)\neq 0$ just for a finite number of $n$'s.  If,  for some  $\varepsilon>0$,   
 \begin{equation}
 \label{eq:finiteV}
|u(t,n)|\le C\left(\frac{e}{(2+\varepsilon)n}\right)^n,\quad  n>0,\ t\in\{0,1\},
\end{equation} 
then $u\equiv 0$.
\end{theorem}

\begin{proof} 
We may assume that $V_n=0$ for $n>N$ and for $n<0$.
Consider the bounded operator  $H=\Delta_d+V: \ell^2\to \ell^2$. The solution $u(t,n)$ is then defined by
 $$
 u(\cdot,t)=e^{iHt}u(\cdot, 0)
 $$
 and hence belongs to $\ell^2$ for all $t>0$.
 
The absolutely continuous spectrum of $H : \ell^2 \to \ell^2$ is the segment $[0,4]$, each point with multiplicity $2$.   The continuous spectrum is parametrized naturally by the the unit circle $\T$:
 \[
 \lambda\in [0,4] \ \Rightarrow  \ \lambda= 2-\theta-\theta^{-1}, \textrm{ for some } \theta \in \T.
 \]
   For each $\theta \in \T$ set  
$\lambda(\theta)=  2-\theta-\theta^{-1}$ and denote by  $e^\pm(\theta) = e^\pm(\theta,n)$ the corresponding Jost solutions of the spectral problem 
\begin{equation}
\label{eigeneq}
Hx= \lambda(\theta)x, 
 \end{equation}
 i.e. the  solutions of \eqref{eigeneq} satisfying
 \[
 e^+(\theta, n)= \theta^n, \ \text{for} \  n>N, \ \text{and} \ e^-(\theta,n)= \theta^n \ \text{for} \ n<0.
 \] 
 We refer the reader to \cite{Tbook} and \cite{Mbook} for the precise definition and detailed discussion of Jost solutions.
 
 Except for $\theta=\pm 1$, each of the pairs $\{e^+(\theta), e^+(\theta^{-1})\}$, 
  $\{e^-(\theta), e^-(\theta^{-1})\}$ is a fundamental system of solutions of  \eqref{eigeneq}.
  Hence we have the representations
 \begin{align*}
  e^+(\theta)=a^-(\theta)e^-(\theta)+b^-(\theta)e^-(\theta^{-1}) \\
  e^-(\theta)= a^+(\theta)e^+(\theta)+b^+(\theta)e^+(\theta^{-1}) 
  \end{align*}
 It is known, see e.g. \cite{Tbook}, that $a^\pm$ and $b^\pm$ are rational functions of $\theta$, with no poles on $\T$,
 and  also for $0\leq n\leq N$ the functions $e^\pm(\theta,n)$ are linear combinations of $\theta^j$, 
  $j\in \{-N, -N+1, \ldots \ , 2N\}$.  In particular,
\begin{equation}
\label{eq:jostbasic1}
\lim_{|\theta|\to+\infty}\frac{\log|a^+(\theta)|}{|\theta|}
=\lim_{|\theta|\to+\infty}\frac{\log|b^+(\theta)|}{|\theta|}=0
\end{equation}
and
\begin{equation}
\label{eq:jostbasic2}
\lim_{|\theta|\to+\infty}\frac{\log|e^\pm(\theta,n)|}{|\theta|}=0\qquad\mbox{for }n=0,\ldots,N.
\end{equation}

Consider the function
  \begin{multline*}
  \Phi(\theta, t)= \sum_{-\infty}^\infty u(t,n)e^-(\theta,n)=\sum_{-\infty}^{-1} u(t,n)e^{-}(\theta,n)\\+a^+(\theta)\sum_{0}^\infty u(t,n)e^{+}(\theta,n)+b^+(\theta)\sum_{0}^\infty u(t,n)e^+(\theta^{-1},n).
  \end{multline*}
 
 For all $t\geq 0$ these functions are in $L^2(\T)$. In addition   the first and the third series in the right-hand side converge 
for $|\theta| > 1$ while the second one converges for $|\theta| < 1$.  For $t=0$ and $t=1$ the 
second term also converges for $|\theta|>1$, by the hypothesis \eqref{eq:finiteV}, thus the functions 
$\Phi(\theta, 0)$ and $\Phi(\theta, 1)$ are holomorphic in $\C\setminus \{0\}$ except perhaps at the  
  poles of the functions $a^+$ and $b^+$.  Actually,  by the  basic energy estimate in the next section one can extend this convergence property to $\Phi(\theta,t)$ for all $t\in[0,1]$, see Corollary \ref{rmk:onesided}.

We have
  \begin{multline*}
 -i\frac{\partial \Phi(\theta,t)}{\partial t} =
 \sum_{n=-\infty}^\infty  (Hu)(t,n) e^-(n,\theta)=\sum_{n=-\infty}^\infty u(t,n)(He^{-})(n,\theta)\\ 
=(2-\theta-\theta^{-1}) \Phi(\theta,t).
   \end{multline*}
	Hence
$   \Phi(\theta,t) =e^{i t(2-\theta-\theta^{-1})}  \Phi(\theta,0 )$,
  and  in particular 
   \beq
   \label{relation}
   \Phi(\theta,1) =e^{i (2-\theta-\theta^{-1})}  \Phi(\theta,0),
   \eeq   
 a relation which extends to the whole complex plane outside of $\theta = 0$.

To derive a contradiction to $\Phi \neq 0$, we write $\Phi(\theta, t)$ as
  \begin{multline*}
   \Phi(\theta,t)=\left (\sum_{n<0} u(t,n)\theta^n + \sum_{n=0}^Nu(t,n)e^-(n,\theta)  + b^+(\theta) \sum_{n > N}u(t,n) \theta^{-n}
   \right ) \\
       +a^+(\theta)  \left ( \sum_{n > N}u(t,n) \theta^{n} \right  )  =\mathrel{\mathop:} A(\theta, t)+a^+(\theta) B(\theta,t).
    \end{multline*}
Since $u(t, \, \cdot \,) \in \ell^2$, we clearly have
$$
\lim_{\theta\to \infty} \frac {\log|A(\theta,t)|}{|\theta|}= \lim_{\theta\to \infty} \frac {\log| a^+(\theta,t)|}{|\theta|}=0,
$$
while \eqref{eq:finiteV} yields that $ B(\theta,t)$, $t=0$ and $t=1$, are entire functions of exponential type at most  $(2+\veps)^{-1}$.
Hence, for each $\alpha \in [0,2\pi]$ we have
$$
 \limsup_{r\to\infty} \frac{\log |B(re^{i\alpha},t)|}{r}  \leq  \frac 1 {2+\veps}, \ t\in\{0,1\},
$$
and therefore
\beq
\label{lims}
 \limsup_{r\to\infty} \frac{\log |\Phi(re^{i\alpha},t)|}{ r}  \leq  \frac 1 {2+\veps}, \ t\in\{0,1\}.
\eeq
By \eqref{eq:if}, we also have
\[
\limsup_{r\to\infty} \frac{\log |\Phi(re^{i\alpha},t)|}{ r}  \geq  -\frac 1 {2+\veps}, \ t\in\{0,1\}.
\]
This leads  to a contradiction unless $\Phi\equiv 0$, since according \eqref{relation} we have 
$$
\limsup_{y\to +\infty} \frac{\log |\Phi(iy,1)|}{y} = 1+ \limsup_{y\to +\infty} \frac{\log |\Phi(iy,0)|}{y} 
> 1- \frac 1 {2+\veps}> \frac 1 {2+\veps}.
$$
\end{proof}

It would be of interest to extend this result to the case of potentials with fast decay, not necessarily compactly supported.

\section{First energy estimate} \label{sec:energy}
In the remainder of the paper we will follow the ideas of \cite{EKPV}, to prove that a solution of the discrete  Schr\"{o}dinger equation which decays sufficiently fast along both half-axes at two different  moments of time is trivial.

We begin with an energy estimate for solutions of a non-homogeneous initial value problem and show that if the initial data is well-concentrated, the energy cannot spread out too fast.

Given $\alpha >0$ and $t\geq 0$ denote
\[
\psi_\alpha(t)= \{\psi_\alpha(t, n)\}_{n\in\Z}=\{(1+|n|)^{\alpha|n|/ (1+t)} \}_{n\in\Z}
\]

\begin{proposition}\label{pr:ee} 
Let $V=V_1+iV_2$, with $V_1,V_2\,:[0,T]\times\Z\rightarrow\R$  and $V_2$ bounded
and $F\,:[0,T]\times\Z\rightarrow\C$ bounded.
Let $u\,:[0,T]\times\Z\rightarrow\C$, $u\in \mathcal{C}^1\bigl([0,T],\ell^2(\Z)\bigr)$, satisfy
\begin{equation}\label{eq:VF}
\partial_t u(t,n)=i(\Delta u(t,n)+V(t,n)u+F(t,n)).
\end{equation}
Assume that $\{\psi_\alpha(0, n)u(0,n)\}\in \ell^2(\Z)$ for some $\alpha \in (0,1]$. Then, for $T>0$, 
\begin{multline}\label{eq:energy}
\|\psi_\alpha (T,n) u(T,n)\|^2_2 \leq\\ e^{CT}  \left (
   \|\psi_\alpha (0, n) u(0, n)\|^2_{2} +
              \int_0^T \|\psi_\alpha(s, n) F(s, n)\|^2_2\,\mbox{d}s  \right ).
\end{multline}              
\end{proposition}

\begin{proof}
Consider $f(t,n)=\psi_\alpha (t,n)u(t,n)$ and let $H(t)=\|f(t,n)\|_2^2$. We fix $\alpha$ 
till the end of the proof and write $\psi=\psi_\alpha$. 

We will perform several formal computations, assuming that $H(t)$ is finite for all $t\in[0,T]$, and then justify these computations at the end of the proof. 

Define 
$$
\kappa(n,t)=\log \psi(t,n)= \frac{\alpha}{1+t}|n|\log(1+|n|).
$$
Then
\[
\partial_t f=i\Delta(\psi^{-1} f)+iVf+\partial_t\kappa f+i\psi F,\]
which we rewrite as $\partial_t f=\S f+\A f+iVf+i \psi F $, where $\S$ and $\A$ are symmetric and anti-symmetric operators, respectively. Explicitly
\begin{align*}
\S f&=\frac i2\left(\psi\Delta(\psi^{-1}f)-\psi^{-1}\Delta(\psi f)\right)+\partial_t\kappa f,\\
\A f&=\frac i2\left(\psi\Delta(\psi^{-1}f)+\psi^{-1}\Delta(\psi f)\right).
\end{align*}

Denote
\beq
\label{eq:coeff}
a_n=\frac{\psi_{n+1}}{\psi_n}-\frac{\psi_n}{\psi_{n+1}},\quad b_n=\frac{\psi_{n+1}}{\psi_n}+\frac{\psi_n}{\psi_{n+1}}.
\eeq
 We then rewrite
\begin{align}
\label{eq:sn}
(\S f)_n=&-\frac{i}{2}(a_n f_{n+1}-a_{n-1}f_{n-1})+(\partial_t\kappa)_nf_n,     \\
\label{eq:an}
(\A f)_n=&
\frac{i}{2}(b_n f_{n+1}+b_{n-1}f_{n-1})-2if_n,
\end{align}
In  what follows we will use the notation  $a_n=a(t,n)$, and similarly for $\psi_n$, et cetera.

We want to control the growth of $H(t)$. Clearly, $\partial_t H(t)=2\Re\langle \partial_t f,f\rangle$ and  thus
\begin{align*}
\partial_t H(t)=&2\langle \S f,f\rangle-2\Im\langle Vf,f\rangle-2\Im\langle \psi F,f\rangle\\
=&2\Im\sum_n a_nf_{n+1}\overline{f}_n
+2\langle \partial_t\kappa f,f\rangle-2\langle V_2f,f\rangle-2\Im\langle \psi F,f\rangle.
\end{align*} 
This implies
$$
\partial_t H(t)\le 2\|\psi F\|_2\|f\|_2+\|V_2\|_\infty\|f\|_2+\sum_n(2\partial_t\kappa_n+|a_n|+|a_{n-1}|)|f_n|^2. 
$$

Our aim is to prove that for all $n\in\Z$
\begin{equation}\label{eq:psipos}
2\partial_t\kappa_n+|a_n|+|a_{n-1}|\le 2C,
\end{equation}
where $C$ is a constant. 
We have 
\[
\partial_t \kappa_n=-\frac{\alpha}{(1+t)^2}|n|\log (|n|+1). 
 \]
Further, $|a_n|\le e^\alpha(|n|+1)^{\alpha}.$ Hence, if
$\alpha\le 1$ we obtain (\ref{eq:psipos}), for $t\in[0,1]$.

Therefore $\partial_t \|f\|_2\le C\|f\|_2+\|\psi F\|_2$ and \eqref{eq:energy}  follows.

In order to justify these computations we truncate 
the weight function $\psi$ to an interval $[-N, N]$:
\[
\psi_N(n,t)=\begin{cases}(|n|+1)^{(1+t)^{-1}\alpha|n|},\ |n|\le N\\
(|N|+1)^{(1+t)^{-1}\alpha|N|},\ |n|>N.\end{cases}
\]
Since the solution $u$ is in $\ell^2$, the relevant norms weighted by $\psi_N$ are guaranteed to be finite and by running the above argument we obtain (\ref{eq:psipos}) and (\ref{eq:energy}) for the weight $\psi_N$, this time rigorously. The desired inequality follows by passing to the limit as $N\to\infty$.
\end{proof}

\begin{corollary}\label{rmk:onesided}
Let $u:[0,1]\times\Z\rightarrow\C$ be a strong solution of the Schr\"odinger equation
\[
\partial_t u(t,n)=i(\Delta u(t,n)+V(t,n)u),
\]
where $V=V_1+iV_2$ is as above. Further suppose that
 \[
 \sum_{n>0}n^{2\alpha n}|u(0,n)|^2 < \infty
 \] 
 for some $\alpha\le 1.$ Then for each $t\in[0,1]$ we have
$$
\sum_{n>0} n^{\alpha n }|u(t,n)|^2 < \infty.
$$
\end{corollary}

\begin{proof}
Define $\tilde{u}(t,n)=0$ for $n<0$ and $\tilde{u}(t,n)=u(t,n)$ for $n\ge 0$. Then $\tilde{u}$ satisfies  \eqref{eq:VF} with $F(t,n)$  bounded and vanishing for $n\not\in\{-1,0\}$. If we apply Proposition \ref{pr:ee} to $\tilde{u}$ we obtain the required estimate
\end{proof}

\section{Logarithmic convexity of weighted $\ell^2$-norms}
\subsection{Preliminary discussion}
From now on we fix $\gamma_0 > 0$ and suppose that $V\,:[0,T]\times\Z\rightarrow\R$ is bounded. Further, we assume that $u$ is a strong solution of
$$
\partial_t u=i(\Delta_d u+Vu)
$$
such that $\|(1+|n|)^{\gamma_0(1+|n|)}u(0,n)\|_2$ and $\|(1+|n|)^{\gamma_0(1+|n|)}u(1,n)\|_2$ are finite.

Following the ideas of \cite{EKPV}, we are looking for a weight  
\beq\label{eq:weight}\psi(t,n)=\exp(\kappa(t,n))\eeq 
to give us a logarithmically convex function $e^{-Ct(1-t)}H(t)$, where 
\[H(t)=\|\psi(t,n)u(t,n)\|_2^2\] and $C$ depends on $V$ and $\psi$. 

We will first use such a convexity argument to show that for any $0<\gamma<\gamma_0$ and any $t\in[0,1]$,
\begin{equation}
\label{eq:normfin}
\|(1+|n|)^{\gamma(1+|n|)}u(t,n)\|_2 < \infty.
\end{equation}
 This also implies that 
\begin{equation}
\label{eq:normfin2}
\|(C_0+|n|+R_0t(1-t))^{\gamma\bigr(C_0+|n|+R_0t(1-t)\bigl)}u(t,n)\|_2<+\infty
\end{equation}
for any $C_0, R_0>0$ and $t\in[0,1]$, and we then set out to prove the logarithmic convexity in $t$ of this latter norm.

In both steps we consider weights of the form \eqref{eq:weight}, with
\[\kappa(t,n)=\gamma(|n|+R(t))\ln^b\bigl(|n|+R(t)\bigr)\]
 where either $1/2 < b < 1$ and $R(t) = 1$, or $b = 1$ and $R(t)=C_0+R_0t(1-t)$. As before we set $f(t,n)=\psi(t,n)u(t,n)$. 

We will first assume that $b<1$, prove estimates independent of $b$, and let $b\to1$ to establish \eqref{eq:normfin}.
This will allow us to justify the computations involved in the second step, when $b=1$ and we prove the convexity of \eqref{eq:normfin2}.
\subsection{Formal computations}
We collect here a number of formal identities which we need in the sequel. 
The first identities are the same as in the continuous case, found in for example \cite{EKPV2},
others are specific to the discrete case. 
We use the notation established in the proof of Proposition \ref{pr:ee}. 

 We already know that
$\partial_t H(t)=2\langle \S f,f\rangle$, since $V$ is real-valued, and thus
\begin{align*}
\partial^2_t H(t)&=2\langle \S_tf,f\rangle+4\Re\langle \S f, f_t\rangle\\&=2\langle \S_tf,f\rangle+4\|\S f\|^2+2\langle [\S,\A]f,f\rangle+4\Re\langle \S f, iVf\rangle\\
&=2\langle \S_tf,f\rangle+2\langle [\S,\A]f,f\rangle+4\Re\langle \S f+iVf,\S f\rangle\\&=
2\langle \S_tf,f\rangle+2\langle [\S,\A]f,f\rangle+
\|2\S f+iVf\|^2-\|Vf\|^2.
\end{align*} 
Therefore we obtain that
\begin{align*}
\|f\|^{2}\partial^2_t(\log H(t)) &=\frac{\|2\S f+iVf\|^2\|f\|^2-4|\langle \S f,f\rangle|^2}{\|f\|^2} \\ &\qquad +2(\langle \S_t f,f\rangle+\langle[\S,\A]f,f\rangle)-\|Vf\|^2\\
&=\frac{\|2\S f+iVf\|^2\|f\|^2-|\Re\langle 2\S f+iVf,f\rangle|^2}{\|f\|^2}\\&\qquad+2(\langle \S_t f,f\rangle+\langle[\S,\A]f,f\rangle)-\|Vf\|^2\\
&\ge 2(\langle \S_t f,f\rangle+\langle[\S,\A]f,f\rangle)-\|Vf\|^2.
\end{align*}
We reiterate that our aim is to show that 
\beq
\label{eq:main1}
\partial_t^2\log H(t)\geq -2C
\eeq
 for some $C\geq 0$, which implies the $\log$-convexity
of $e^{-Ct(1-t)}H(t)$. The last term in the right-hand side above is clearly bounded 
below by $-C\|f\|^2$ since $V$ is bounded. It suffices to establish an
estimate of the first two terms of the form
\beq
\label{eq:main}
\langle \S_t f,f\rangle+\langle[\S,\A]f,f\rangle > - C\|f\|^2.
\eeq

We refer now to \eqref{eq:an}, \eqref{eq:sn}. It follows that 
\[
(\S_tf)_n =-i/2(a'_n f_{n+1}-a'_{n-1}f_{n-1})+\kappa''_nf_n,
\]
 and  finally
$$(2\S_tf+2[\S,\A]f)_n=\nu_{n+1}f_{n+2}+\lambda_nf_{n+1}+\mu_nf_n+\overline{\lambda_{n-1}}f_{n-1}+\nu_{n-1}f_{n-2},$$
where 
\begin{align*}
\nu_{n+1} &=\frac12(a_nb_{n+1}-a_{n+1}b_n), \\
\lambda_n &=-ib_n(\kappa_{n+1}'-\kappa_n')-ia_n',\\
\mu_n &=a_nb_n-a_{n-1}b_{n-1}+2\kappa_n'',
\end{align*}
and, as before, the coefficients $a_n$,$b_n$ are defined in  \eqref{eq:coeff}.
 
Clearly $\psi_n'=\kappa_n'\psi_n$, implying that
$a_n'=(\kappa_{n+1}'-\kappa_n')b_n$ and 
\[\lambda_n=-2ib_n(\kappa_{n+1}'-\kappa_n').\]

\subsection{Estimates with an auxiliary weight}

\begin{proposition}
Let $\gamma > 0$. Assume that $u$ is a strong solution of
$$
\partial_t u=i(\Delta_d u+Vu)
$$
where the potential $V$ is a bounded real-valued function. Let also  
\begin{equation}
\label{eq:prop41}
\norm{(1+|n|)^{\gamma(1+|n|)}u(t,n)}_2<+\infty,\quad t\in\{0;1\}.
\end{equation}
 Then, for all $t\in[0,1]$,
$\norm{(1+|n|)^{\gamma(1+|n|)}u(t,n)}_2<+\infty$.
\end{proposition}

\begin{proof}
Consider the weight function
\[
\psi(n)= e^{\kappa_b(n)}, \  \kappa_b(n)=\gamma(1+|n|)\ln^b\bigl(1+|n|\bigr),
\]
where $1/2<b<1$.  
Note that the hypotheses \eqref{eq:prop41} combined with Proposition \ref{pr:ee}
show that $H_b(t)=\|\exp(\kappa_b(n))u(t,n)\|_2^2$ is finite for all $t$, allowing us to justify 
the computations of the preceding section for this choice of weight. We will show that $H(t)=H_b(t)$ satisfies
\eqref{eq:main1}
 with some  $C$ independent of    $b$, whence
\begin{multline*}
\|\exp(\kappa_b(n))u(t,n)\|_2^2\le e^{\frac{C}{2}t(1-t)}H_b(0)^{1-t}H_b(1)^t\\
\le e^{\frac{C}{2}t(1-t)}\norm{(1+|n|)^{\gamma(1+|n|)}u(0,n)}_2^{2(1-t)}\norm{(1+|n|)^{\gamma(1+|n|)}u(1,n)}_2^{2t}.
\end{multline*}
Letting $b\to 1$ and applying the monotone convergence theorem then concludes the proof.

We refer to  computations in  the previous section.
 In the current setting $S_t=0$ and $\lambda_n=0$ so relation \eqref{eq:main}   reduces  to
 \beq
 \label{eq:wb}
\scal{2[\S,\A]f,f}\ge -C\|f\|^2.
\eeq

We have 
\[
\scal{2[\S,\A]f,f}=\sum_n\mu_n|f_n|^2+2\Re\sum_n \nu_{n+1}f_{n+2}\overline{f_n},\]
 where
\[
\mu_n=a_nb_n-a_{n-1}b_{n-1}=\frac{\psi_{n+1}^2}{\psi_n^2}-\frac{\psi_{n}^2}{\psi_{n-1}^2}-\frac{\psi_{n}^2}{\psi_{n+1}^2}+\frac{\psi_{n-1}^2}{\psi_n^2},\]
and
\[
\nu_{n+1}=\frac12(a_nb_{n+1}-a_{n+1}b_n)=-\frac{\psi_n\psi_{n+2}}{\psi_{n+1}^2}+\frac{\psi_{n+1}^2}{\psi_{n+2}\psi_n},
\]
where the coefficients $a_n$ and $b_n$ are defined in \eqref{eq:coeff}.
By appealing to the second derivative of $x \mapsto (1+x)\ln^b(1+x)$ it is easy to verify that $\kappa_b(n+2)+\kappa_b(n)-2\kappa_b(n+1)$ is always non-negative and uniformly bounded from above.  Thus $\nu_{n+1}$ is uniformly bounded and $\mu_n\ge 0$. This implies \eqref{eq:wb}.
\end{proof}
\subsection{Convexity estimate}
In this subsection we consider the weight function given by
\[
\psi(t,n)=e^{\kappa(t,n)}, \ \mbox{where} \ \kappa(t,n)=\gamma(|n|+R(t))\ln(|n|+R(t)),
\] 
and  $R(t)=C_0+R_0t(1-t)$, $R_0 > 0$,  $C_0$  being large enough. 
As before we define $H(t)=\|u(t,n)\psi(t,n)\|_2^2$.

\begin{lemma} \label{lem:convex}
For every $\gamma>(3+\sqrt{3})/2$  there exists $C(\gamma)$  such that for $C_0>C(\gamma)$ and $R(t)=C_0+R_0t(1-t)$ we have
\[
\partial_t^2(\log H(t))\ge -\frac{4\gamma}{2\gamma-3} R_0\log R_0- C_1 R_0-C_2,\]
where $C_1$ and $C_2$ depend on $\gamma$ and $\|V\|_\infty$ only.
\end{lemma}

\begin{proof}
For $n\ge 0$ we have 
\[\frac{\psi(t,n+1)}{\psi(t,n)}=(n+1+R(t))^\gamma\left(1+\frac1{n+R(t)}\right)^{\gamma(n+R(t))}, \]
and $\psi_{n}=\psi_{-n}$ for $n< 0$. Hence $a_{n}=-a_{-n-1}$ and $b_{n}=b_{-n+1}$ for $n<0$, which in turn implies that $\mu_{n}=\mu_{-n}$ and $\lambda_n=-\lambda_{-n-1}$  when $n<0$. We have also $\mu_0=2a_0b_0+2\kappa_0''$.

As before, we get
\[
|\nu_{n+1}|=\left|\frac{\psi_{n+1}^2}{\psi_n\psi_{n+2}}-\frac{\psi_n\psi_{n+2}}{\psi_{n+1}^2}\right|\le C_3,\]
where $C_3$ depends on $\gamma$ only.

Let $\phi(M)=\gamma M\ln M$ and $M=M(t,n)=|n|+R(t)$. In this notation we have for $n\neq 0$
\[
\mu_n\ge \exp(2\phi(M+1)-2\phi(M))-\exp(2\phi(M)-2\phi(M-1))  -C_4+2\kappa_n'',\ 
\]
where $C_4$ is a constant that depends only on $\gamma$.  
 The derivatives of $\kappa_n$ are
\[ 
\kappa'_n(t)=R'(t)\phi'(|n|+R(t)),
\]
\[
\kappa''_n(t)=-2R_0\phi'(|n|+R(t))+(R'(t))^2\phi''(|n|+R(t)).\]
Then, by the Taylor expansions, we obtain that, for each $\epsilon>0$ and  $C_0=C_0(\epsilon)$ large enough, 
\begin{multline*}
\mu_n\ge 2\gamma e^{2\gamma}M^{2\gamma-1}+\gamma e^{2\gamma}\left(\frac{(\gamma-1)^2}3-\epsilon\right)M^{2\gamma-3} \\ +2A^2\gamma M^{-1}-4R_0\gamma(1+\ln M) - C_4,
\end{multline*}
where $A=|R'(t)|$ and $n\neq 0$. Futher,
\[
\mu_0\ge (2-\epsilon) M^{2\gamma}e^{2\gamma}+2A^2\gamma M^{-1}-4R_0\gamma(1+\ln M) - C_4.\]

We introduce the notation
\[
\sigma_n=2\gamma e^{2\gamma}M^{2\gamma-1}+\gamma e^{2\gamma}\left(\frac{(\gamma-1)^2}3-2\epsilon\right)M^{2\gamma-3}+2A^2\gamma M^{-1},\]
and 
\[
\rho_n=\epsilon\gamma e^{2\gamma}M^{2\gamma-3}-4R_0\gamma(1+\ln M)
\]
so that $\mu_n\ge \sigma_n+\rho_n-C_4$ for all $n$.
Note that by the inequality of arithmetic and geometric means we have
\[\sigma^2_n\ge 8A^2\gamma^2e^{2\gamma}\left(2M^{2\gamma-2}+\left(\frac{(\gamma-1)^2}3-2\epsilon\right)M^{2\gamma-4}\right). \]

For $n\ge 0$ we have also
\[
|\kappa_{n+1}'-\kappa_n'|= |R'(t)| (\phi'(M+1)-\phi'(M))=A\gamma\ln(1+M^{-1}).
\]
Hence, for sufficiently large $C_0$,
\begin{multline*}
|\lambda_n|=2|(\kappa_{n+1}'-\kappa_n')||b_n|
\le  2A\gamma e^{\gamma}M^{\gamma-1}+A\gamma e^\gamma(\gamma-1)M^{\gamma-2}\\+A\gamma e^{\gamma}\left(\frac{3\gamma^2-10\gamma+8}{12}+\epsilon\right)M^{\gamma-3},\ n\ge 0.
\end{multline*}

To estimate $\partial_t^2(\log H(t))$ we note that
\begin{multline*}
\scal{\S_tf+2[\S,\A]f,f}=\sum_n\mu_n|f_n|^2+2\Re\sum_n \nu_{n+1}f_{n+2}\overline{f_n}+2\Re \sum_n\lambda_n f_{n+1}\overline{f_n}\\
\ge\sum_{n} \sigma_n|f_n|^2+2\Re\sum_{n}\lambda_n f_{n+1}\overline{f_n}+\sum_{n} \rho_n|f_n|^2-(C_3+C_4)\sum_n|f_n|^2
\end{multline*}
First, we consider the first two terms. If we show that for any $x,y \geq 0$ 
\begin{equation}\label{eq:quadr}
\sigma_{n}x^2+\sigma_{n+1}y^2\ge 4|\lambda_n|xy,
\end{equation}
then the summation of these inequalities  with $x=f_n$, $y=f_{n+1}$ yields
\[ 
\sum_{n} \sigma_n|f_n|^2+2\Re\sum_{n}\lambda_n f_{n+1}\overline{f_n}\ge 0.\]
To show \eqref{eq:quadr} we have to check that 
\begin{equation} \label{eq:quadr2}
 \sigma_n\sigma_{n+1}\ge 4|\lambda_n|^2, \quad n\ge 0.
 \end{equation}
 Actually we show \eqref{eq:quadr} only for $n\ge 0$. The relations for negative integers given in the beginning of the proof then imply the inequality for all $n$. 
 
  Using the estimates above, we have
\begin{multline*}
\sigma_n^2\sigma_{n+1}^2\ge 64A^4\gamma^4e^{4\gamma}\left(4M^{4\gamma-4}+8(\gamma-1)M^{4\gamma-5}\right)\\
+64A^4\gamma^4e^{4\gamma}\left(4(\gamma-1)(2\gamma-3)+4\left(\frac{(\gamma-1)^2}3-2\epsilon\right)\right)M^{4\gamma-6}.
\end{multline*}
While
\begin{multline*}
16|\lambda_n|^4\le 64A^4\gamma^4e^{4\gamma}\left(4M^{4\gamma-4}+8(\gamma-1)M^{4\gamma-5}\right)\\+4A^4\gamma^4e^{4\gamma}\left(
6(\gamma-1)^2
+8\left(\frac{3\gamma^2-10\gamma+8}{12}+\epsilon\right)\right)M^{4\gamma-6}.
\end{multline*}
Inequality \eqref{eq:quadr2} hence follows for sufficiently small $\epsilon$ when
\[
2(\gamma-1)(2\gamma-3)+\frac{2(\gamma-1)^2}3>3(\gamma-1)^2+\frac{3\gamma^2-10\gamma+8}3.
\]
The last inequality is equivalent to $2\gamma^2-6\gamma+3>0$, which holds for $\gamma>(3+\sqrt{3})/2$.

Finally  by minimizing in $M$ one obtains that, for $\gamma>3/2$, 
\[
\rho_n\ge
\min_{M>0}\{\epsilon\gamma e^{2\gamma}M^{2\gamma-3}-4R_0\gamma(1+\ln M)\}\ge
-\frac{4\gamma}{2\gamma-3}R_0\ln R_0- C_1 R_0,\]
where $C_1$ depends on $\gamma$ and $\epsilon$. The conclusion of the lemma follows.
\end{proof}


\subsection{Concluding arguments}
Using the weight $\psi(n,t,R_0)$ from the last section and Lemma \ref{lem:convex}, we obtain that 
\[H_{R_0}(t)\exp(-d(R_0,\gamma)t(1-t))\]
 is logarithmically convex, where $$d(R_0,\gamma)=\frac{2\gamma}{2\gamma-3}R_0\ln R_0+ \frac{C_1}{2}R_0+\frac{C_2}{2}.$$  Hence, for $t=1/2$ we obtain
\[
H_{R_0}(1/2) \le \exp\left(\frac{\gamma}{2(2\gamma-3)}R_0\ln R_0+\frac{C_1}{8}R_0+\frac{C_2}{8}\right)H_{R_0}(0)^{1/2}H_{R_0}(1)^{1/2}.\]
But since $R(0)=R(1)=C_0$ we see that $H(0)$ and $H(1)$ do not depend on the choice of $R_0$. We obtain that
\begin{multline*}
|u(1/2,n)|^2\exp(2\gamma(|n|+C_0+R_0/4)\ln(|n|+C_0+R_0/4))\\\le
D\exp\left(\frac{\gamma}{2(2\gamma-3)}R_0\ln R_0+\frac{C_1}{8}R_0\right),
\end{multline*}
where $D$ is a constant independent of $n$ and $R_0$. However, this last inequality is clearly impossible for large $R_0$ when $\gamma>2$, unless $u(1/2, \, \cdot \,) \equiv 0$, which of course implies that $u \equiv 0$.

Our work of this section can be summarized as follows.
\begin{theorem} \label{thm:main}
 Assume that  $\gamma>(3+\sqrt{3})/2$ and that $V(t,n)$ is a real-valued bounded function. If $u$ is a strong solution of
$$
\partial_t u=i(\Delta_d u+Vu)
$$
such that \[\norm{(1+|n|)^{\gamma(1+|n|)}u(0,n)}_2,\norm{(1+|n|)^{\gamma(1+|n|)}u(1,n)}_2<+\infty,\]
then  $u\equiv 0$.
\end{theorem}

 {\em Remark.} This result is most likely not sharp. The authors  expect that  a  milder decay condition (with $\gamma=1+\epsilon$) and even just one-sided decay  should imply  uniqueness as in the case of free Schr\"odinger evolution.

\end{document}